\documentclass[reqno]{amsart}
\usepackage[utf8x]{inputenc}
\usepackage{amsmath, amsthm, amssymb}
\usepackage[usenames,dvipsnames,svgnames,table]{xcolor}
\usepackage[margin=1.3in]{geometry}

\newtheorem{theorem}{Theorem}[section]
\newtheorem{proposition}[theorem]{Proposition}
\newtheorem{lemma}[theorem]{Lemma}

\newtheorem{corollary}[theorem]{Corollary}

\newcommand{\R}{\mathbb R}

\newcommand{\eps}{\varepsilon}

\newcommand{\dd}{\, \mathrm{d}}

\newcommand{\tr}{\mbox{tr}}

\newcommand{\intt}{\scriptsize{\mbox{int}}}
\newcommand{\ext}{\scriptsize{\mbox{ext}}}
\newcommand{\midd}{\scriptsize{\mbox{mid}}}

\numberwithin{equation}{section}
\numberwithin{theorem}{section}
\title{Gaussian bounds for the inhomogeneous Landau equation with hard potentials}

\author{Stanley Snelson}
\address{Department of Mathematical Sciences, Florida Institute of Technology, 150 W. University Blvd., Melbourne, FL 32901}
\email{ssnelson@fit.edu}
\thanks{The author was partially supported by a Ralph E. Powe Award from ORAU, and would also like to thank Luis Silvestre for several helpful discussions regarding this work.}

\begin{document}

\begin{abstract}
We consider weak solutions of the spatially inhomogeneous Landau equation with hard potentials ($\gamma \in (0,1]$), under the assumption that mass, energy, and entropy densities are under control. In this regime, with arbitrary initial data, we show that solutions satisfy pointwise Gaussian upper and lower bounds in the velocity variable. This is different from the behavior in the soft potentials case ($\gamma <0$), where Gaussian estimates are known not to hold without corresponding assumptions on the initial data. Our upper bounds imply weak solutions are $C^\infty$ in all three variables, and that continuation of solutions is governed only by the mass, energy, and entropy.
\end{abstract}

\maketitle

\section{Introduction}

The Landau equation is an integro-differential kinetic model arising in plasma physics. For $(t,x,v)\in [0,T]\times \R^d\times\R^d$, the solution $f(t,x,v) \geq 0$ satisfies
%\begin{align}
%\partial_t f + v\cdot \nabla_x f &= \nabla_v \cdot \left( \int_{\R^d} a(v-w)[f(w)\nabla f(v) - f(v)\nabla f(w)] \dd w \right),\label{e:main}\\
%a(z) &= a_{d,\gamma}|z|^{\gamma+2}\left(I - \frac z {|z|}\otimes \frac z{|z|}\right),\nonumber
%\end{align}
\begin{equation}\label{e:main}
\partial_t f + v\cdot \nabla_x f = \nabla_v\cdot \left[\bar a^f(t,x,v)\nabla_v f\right] + \bar b^f(t,x,v)\cdot\nabla_v f + \bar c^f(t,x,v) f,
\end{equation}
with the nonlocal coefficients $\bar a^f(t,x,v)\in \R^{d\times d}$, $\bar b^f(t,x,v) \in \R^d$, and $\bar c^f(t,x,v)\in \R$ defined by
\begin{align}
\bar a^f(t,x,v) &:= a_{d,\gamma}\int_{\R^d} \left( I - \frac w {|w|} \otimes \frac w {|w|}\right) |w|^{\gamma + 2} f(t,x,v-w) \dd w,\label{e:a}\\
\bar b^f(t,x,v) &:= b_{d,\gamma}\int_{\R^d} |w|^\gamma w f(t,x,v-w)\dd w,\label{e:b}\\
\bar c^f(t,x,v) &:= c_{d,\gamma}\int_{\R^d} |w|^\gamma f(t,x,v-w)\dd w, \label{e:c}
\end{align}
and $a_{d,\gamma}$, $b_{d,\gamma}$, $c_{d,\gamma}$ are positive constants. For $\gamma > -d$, the Landau equation arises as a limit of the Boltzmann equation as grazing collisions predominate. %(See \cite{alexandre2004landau}.) 
Although the case $d=3$, $\gamma = -3$ (where \eqref{e:c} must be replaced by $\bar c^f = c f$) is the most directly relevant as a physical model (see e.g. \cite{chapmancowling, lifschitzpitaevskii}), the cases $\gamma \in (-d,1]$ have also attracted a great deal of interest from a mathematical standpoint, partly motivated by the desire to better understand grazing collisions, which play a delicate role in the study of the non-cutoff Boltzmann equation. In this article, we are interested in the case of \emph{hard potentials}, i.e. $\gamma \in(0,1]$. (In fact, our results hold for any $\gamma \in (0,2)$, but $\gamma \in (0,1]$ is the case of interest in the literature.)

Functions of the form $c e^{-\alpha|v-v_0|^2}$ for $v_0\in \R^d$ and $c,\alpha >0$ (referred to as \emph{Maxwellians}) are equilibrium solutions of \eqref{e:main}. The goal of this article is to prove, under relatively weak \emph{a priori} assumptions, that solutions of \eqref{e:main} are bounded above and below by Maxwellians, and that these estimates depend only on physically meaningful quantities.  
%
%Furthermore, we look for estimates with constants that depend only on physically meaningful quantities. 
Define 
\begin{align*}
M_f(t,x) &= \int_{\R^d} f(t,x,v)\dd v, &&\mbox{(mass density)}\\
E_f(t,x) &= \int_{\R^d} |v|^2 f(t,x,v)\dd v, &&\mbox{(energy density)}\\
H_f(t,x) &= \int_{\R^d} f(t,x,v)\log f(t,x,v) \dd v. &&\mbox{(entropy density)}
\end{align*}
These hydrodynamic quantities corresponding to $f$ are physically observable at the macroscopic scale. We will assume throughout that for all $t\in [0,T]$ and $x\in \R^d$,
\begin{equation}\label{e:hydro}
 0< m_0 \leq M_f(t,x) \leq M_0, \quad E_f(t,x) \leq E_0, \quad \mbox{and} \quad H_f(t,x) \leq H_0,
 \end{equation}
for some constants $m_0$, $M_0$, $E_0$, $H_0$.  %It can be shown that $\|M_f(t,\cdot)\|_{L^1_x(\R^d)}$ and $\|E_f(t,\cdot)\|_{L^1_x(\R^d)}$ are conserved by \eqref{e:main}, and that $\|H_f(t,\cdot)\|_{L^1_x(\R^d)}$ is nonincreasing, but it is not known whether the property \eqref{e:hydro} is conserved by the equation (in the spatially inhomogeneous case), % the equation preserves solutions satisfying \eqref{e:hydro} ini. 
In the spatially homogeneous case (when $f$ is independent of $x$), the mass, energy, and entropy are functions of $t$ only, and it is known that mass and energy are conserved, and entropy is nonincreasing. These properties follow from the integral identities $\int_{\R^d} Q(f,f)\dd v = \int_{\R^d}|v|^2 Q(f,f) \dd v = 0$ and $\int_{\R^d}\log f Q(f,f) \dd v \leq 0$. (See, e.g. \cite[Chapter 1, Section 2]{villani2002review}). Therefore, assumption \eqref{e:hydro} would be unnecessary in the homogeneous case, as long the initial data has finite mass, energy, and entropy. But in the spatially inhomogeneous case considered here, the above integral identities only imply that the $L^1_x$ norms of $M_f$ and $E_f$ are conserved, and that the $L^1_x$ norm of $H_f$ is nonincreasing. It is not clear whether these hydrodynamic quantities can blow up in $L^\infty_x$ in finite time. Therefore, we include \eqref{e:hydro} as an assumption. We say a constant is \emph{universal} if it depends only on $d$, $\gamma$, $m_0$, $M_0$, $E_0$, and $H_0$. 

We will work with solutions satisfying 
\begin{equation}\label{e:G0} 
G_f(t,x) := \int_{\R^d} |v|^{\gamma+2} f(t,x,v) \dd v \leq G_0,
\end{equation}
for some $G_0  > 0$. This ensures that the matrix $\bar a^f$ in \eqref{e:a} is well-defined, but we will seek estimates that do not depend quantitatively on $G_0$. % For the right-hand side of \eqref{e:main} to be well-defined, we need to make the additional assumption that $G_f(t,x)$ is finite at every $t$ and $x$. For technical reasons, we assume 
Assumption \eqref{e:G0} would clearly be unnecessary in the case $\gamma \leq 0$.

We say that $f\geq 0$ satisfying \eqref{e:hydro} and \eqref{e:G0} is a weak solution of \eqref{e:main} if $f$, $\nabla_v f$, and $(\partial_t + v\cdot \nabla_x)f$ are in $L^2_{loc}([0,T]\times \R^{2d})$, and 
\[\int_{[0,T]\times \R^{2d}} \phi (\partial_t+v\cdot \nabla_x)f \dd v \dd x \dd t = \int_{[0,T]\times \R^{2d}} \left[-\bar a^f_{ij} \partial_{v_i} \phi \partial_{v_j} f + \phi(\bar b^f \cdot \nabla_v f + \bar c^f f)\right]\dd v \dd x \dd t.\]
for any $\phi \in H_0^1([0,T] \times \R^{2d})$, where repeated indices are summed over.

Our first main result gives pointwise Gaussian upper bounds for weak solutions:
\begin{theorem}\label{t:main}
	Let $\gamma \in (0,1]$, and let $f$ be a bounded weak solution to the Landau equation \eqref{e:main} on $[0,T]\times \R^{2d}$, satisfying \eqref{e:hydro} and \eqref{e:G0}. Then there exists a decreasing function $K: \R_+ \to \R_+$ with $K(t) \to \infty$ as $t\to 0+$, such that 
%	\[f(t,x,v) \leq C \left(e^{Ct^{-2/\gamma}} +1\right)e^{-\alpha |v|^2},\]
	\[f(t,x,v) \leq K(t) e^{-\alpha |v|^2}, \quad (t,x,v) \in (0,T]\times \R^{2d},\]
	with $K(t)$ and $\alpha>0$ depending on universal constants and $G_0$. %$d$, $\gamma$, $m_0$, $M_0$, $E_0$, $H_0$, and $G_0$.
	
	Furthermore, there exist a decreasing function $J$ and an increasing function $\beta$ from $\R_+\to \R_+$ with $J(t) \to \infty$ and $\beta(t) \to 0$ as $t\to 0+$, such that  
%	  \[f(t,x,v) \leq C \left(e^{Ct^{-2/\gamma}} +1\right)e^{-\beta t^?|v|^2},\]
	 \[f(t,x,v) \leq J(t) e^{-\beta(t)|v|^2}, \quad (t,x,v) \in (0,T]\times \R^{2d},\]
	  with $J(t)$ and $\beta(t)$ depending on universal constants. In particular, $J(t)$ and $\beta(t)$ are independent of $G_0$.
\end{theorem}
Explicit expressions for $K(t)$, $J(t)$, and $\beta(t)$ are given below. Note that all three functions are independent of the time of existence $T$. %We use ``decreasing'' and ``increasing'' in the non-strict sense. %In fact, $\beta(t)$ is bounded away from zero on $\{t \geq t_0\}$ for any $t_0>0$.

%(Compare to moderately soft)    Another advantage of our theorem over the case $\gamma \in (-2,0]$ covered in \cite{cameron2017landau} is that we do not require our solutions to be \emph{a priori} bounded in $L^\infty_v$.

%The strategy of the proof is as follows: first, following the approach of \cite{cameron2017landau}, we apply rescaled versions of the local $L^\infty$ estimates derived in \cite{golse2016}, and track how the constants behave as $|v|\to \infty$ to obtain a global upper bound in terms of $m_0$, $M_0$, $E_0$, $H_0$, and $G_0$. %This requires a change of variables in $v$ to deal with the degeneration of the ellipticity constants of the matrix $a_{ij}$ as $|v|\to \infty$.
%This gives pointwise polynomial decay in $v$, but in this case the upper bound depends on $G_0$. 
%Next, using a maximum principle argument, we prove Gaussian decay in $v$, with estimates depending on the same constants. %So far, all estimates involve constants depending on $G_0$, and we track this dependence carefully. 
%Finally, we use the Gaussian decay and the energy bound to obtain a small improvement on the bound of $G_f(t,x)$, and iterate this improvement to obtain a bound for 
The key step in the second statement of Theorem \ref{t:main} is finding an upper bound for $G_f(t,x)$ that is independent of $G_0$ (Theorem \ref{t:gamma2}). Since this upper bound does not depend on the initial data, it must blow up as $t\to 0$, which is why $\beta(t)$ in Theorem \ref{t:main} degenerates as $t\to 0$. 

%This bound on $G_f(t,x)$, combined with the estimate of \cite{golse2016} implies $f$ is H\"older continuous in its entire domain. 
By the estimate of \cite{golse2016}, $f$ is locally H\"older continuous in its entire domain. The Gaussian bounds of Theorem \ref{t:main} allow us to pass regularity of $f$ to regularity of the nonlocal coefficients (see \eqref{e:a} and \eqref{e:c}) and bootstrap Schauder estimates exactly as in \cite{henderson2017smoothing} %, Gaussian decay allows us to pass the regularity of $f$ to the regularity of the coefficients, and bootstrap Schauder estimates %of \cite{henderson2017smoothing} %(which require Gaussian decay in $v$) 
to conclude the solution is smooth:
%The bounds on $G_f(t,x)$ provided by Theorem \ref{t:gamma2} imply the coefficients in the divergence form of the Landau equation (\eqref{e:divergence} below) are bounded independently of $G_0$ on $[t_0,T]$ for any $t_0>0$. The main theorem of \cite{golse2016} then implies $f$ is H\"older continuous 
\begin{corollary}\label{c:smooth}
	Any bounded weak solution $f$ of \eqref{e:main} on $[0,T]\times \R^{2d}$, satisfying \eqref{e:hydro} and \eqref{e:G0}, is in $C^\infty((0,T]\times \R^{2d})$. Furthermore, all partial derivatives satisfy Gaussian upper bounds in $v$ that are uniform in $(t,x)\in [t_0,T]\times \R^{d}$ and depend only on universal constants, the order of the derivative, and $t_0$.
\end{corollary}   
%Unlike in the case of soft potentials, this smoothing result does not require any decay assumption on the initial data.
%we obtain uniform bounds  implies Since the Schauder estimates of \cite{henderson2017smoothing} require pointwise Gaussian decay in the velocity variable, 

Theorem \ref{t:main} and Corollary \ref{c:smooth} make no decay assumption on the initial data. This is in contrast to the corresponding results for soft potentials ($\gamma <0$), which require the initial data to satisfy Gaussian decay in $v$. (See \cite{cameron2017landau, henderson2017smoothing}.) Theorem \ref{t:main} extends a result in \cite{desvillettes2000landau} for the hard potentials case of the spatially homogeneous equation, which states that arbitrarily high moments of the solution are finite for $t>0$, with an upper bound that may degenerate as $t\to 0$.
%: Gaussian upper bounds in the initial data are propagated for positive times, but do not hold in general, without some decay assumption on the initial data (see \cite{cameron2017landau}). 
%
%upper bounds proportional to $|v|^{-1}$ hold for any bounded solution satisfying \eqref{e:hydro}, but polynomial decay in $v$ with exponent larger than $d+2$ (and in particular, Gaussian decay) cannot hold in general 
%
% Guassian decay in $v$ and smoothness require the assumption that the initial data has Gaussian decay \cite{cameron2017landau, henderson2017smoothing}. For general initial data, solutions have pointwise decay proportional to $|v|^{-1}$, and \emph{a priori} polynomial decay with power larger than $d+2$ cannot hold.% without any decay assumption on the initial data (see \cite{cameron2017landau}). 
% 

%We also show that Gaussian lower bounds in the initial data are propagated for positive times, with constants depending on $G_0$ (Theorem \ref{t:lower}).

Our last result gives lower Gaussian bounds in $v$. %, but unlike the upper bounds of Theorem \ref{t:main}, the decay constants vary with both $t$ and $x$ in general. However, if the initial data satisfies a Gaussian lower bound that is uniform in $x$, this uniform bound is propagated forward in time (with constants depending on $G_0$). 
Statement (a) is a propagating estimate, and statement (b) is a self-generating estimate analogous to Theorem \ref{t:main}:
\begin{theorem}\label{t:lower} Let $f$ be as in Theorem \ref{t:main}.
	\begin{enumerate}
		\item[(a)]  There exist $\alpha, \mu >0$ depending on universal constants and $G_0$, such that if $f(0,x,v) \geq c e^{-\alpha|v|^2}$ for all $(x,v)\in \R^{2d}$, then 
		\[f(t,x,v) \geq c e^{-\mu t} e^{-\alpha |v|^2}, \quad (t,x,v) \in [0,T]\times \R^{2d}.\]
	%	for all $(t,x,v) \in [0,T]\times \R^{2d}$.
		
	%	\item[(b)] For any $t_0 \in (0,T]$ and $x_0\in \R^d$, there exist $v_0\in \R^d$ and $\delta, \mu, \rho , r >0$ depending on $t_0$, $x_0$, and universal constants, such that 
%		\[ f(t,x,v) \geq \delta e^{-\mu (t-t_0)}e^{-\rho[1+ (t-t_0)^{-1}] |v-v_0|^2}, \quad t > t_0, \, |x-x_0|<r, \, v\in \R^d.\]
	%	with constants depending on 
		
		%(Note that (b) holds with no assumption on the initial data.) 
		
		\item[(b)] There exist an increasing function $\delta$ and a decreasing function $\omega$ from $\R_+$ to $\R_+$ with $\delta(t) \to 0$ and $\omega(t) \to \infty$ as $t\to 0+$, such that 
		\[f(t,x,v) \geq \delta(t) e^{-\omega(t)|v|^2}, \quad (t,x,v) \in (0,T]\times \R^{2d},\]
		with $\delta(t)$ and $\omega(t)$ depending only on universal constants.
		
%		If $f(0,x,v) \geq \delta > 0$ for $v\in B_r(v_0)$ and $x\in \R^d$, then there exist $\mu>0$ and a decreasing function $\omega:\R_+\to \R_+$ with $\omega(t)\to \infty$ as $t\to 0+$, such that %$f$ satisfies lower bounds of the form
%		\[f(t,x,v) \geq \delta e^{-\mu t} e^{-\omega(t)|v-v_0|^2}, \quad (t,x,v) \in (0,T]\times \R^{2d}.\] %for some $\delta'>0$.
%		The function $\omega(t)$ depends on universal constants and $r$, and $\mu$ depends on universal constants, $r$, and $G_0$. %depending on $v_0$ and universal constants
%	%	(Wrong time dependence? $G(t)$?)
	\end{enumerate}
\end{theorem}

Theorem \ref{t:lower}(b) is also false for soft potentials, where the optimal generating lower bounds have decay proportional to $e^{-\alpha |v|^{2+|\gamma|}}$. (See \cite{HST2017landau}.)

\subsection{Related work}

The starting point of conditional regularity for \eqref{e:main} under the assumption \eqref{e:hydro} is the work of Golse-Imbert-Mouhot-Vasseur \cite{golse2016}, who proved a Harnack inequality and local $C^\alpha$ estimate for weak solutions of the Landau equation under the assumption \eqref{e:hydro}. % for a class of divergence-form kinetic equations that includes the linearized Landau equation. 
(See also the earlier, related estimates % were obtained for general classes of ultraparabolic equations with rough coefficients that do not contain Landau, 
by Pascucci-Polidoro \cite{pascucci2004ultraparabolic} and Wang-Zhang \cite{Wang2009ultraparabolic}.) In the case $\gamma >0$, the estimate of \cite{golse2016} depends on a quantitative upper bound for $G_f(t,x)$. %($\gamma+2$) moment of $f$, so our results in particular allow one to remove this dependence in the local estimates. 
In the case of moderately soft potentials ($\gamma \in (-2,0]$), the present author, jointly with Cameron and Silvestre \cite{cameron2017landau}, used the local estimate of \cite{golse2016} to derive global estimates for weak solutions satisfying \eqref{e:hydro}, and found \emph{a priori} pointwise decay proportional to $(1+|v|)^{-1}$. It was also shown in \cite{cameron2017landau} that  polynomial decay in $v$ with exponent greater than $d+2$ cannot hold for arbitrary initial data, and in particular, Gaussian decay in $v$ cannot hold in general. %This is in contrast to the result in this paper for $\gamma \in (0,1]$. 
It is not clear how to bridge this gap between the known and optimal \emph{a priori} decay, but by our Theorem \ref{t:main}, there is no such gap when $\gamma \in (0,1]$. %(Lower bounds?)

In the same context of $\gamma\in (-2,0]$ and weak solutions satisfying \eqref{e:hydro}, $C^\infty$ smoothing in all three variables was established by Henderson-Snelson \cite{henderson2017smoothing}, for initial data with Gaussian decay. A similar result holds for very soft potentials ($\gamma\in [-d,-2]$), %but estimates depend additionally on the $L^\infty$ norm and a higher moment of $f$. 
with an additional higher integrability assumption on $f$. See also Imbert-Mouhot \cite{imbert2018toymodel} for a smoothing result on a related kinetic model. %Even in the spatially homogeneous case, it is not known how to derive $L^\infty$ bounds in terms of physical quantities when $\gamma < -2$. 
In \cite{HST2017landau}, Henderson-Snelson-Tarfulea derived pointwise lower bounds for solutions with mass and energy densities bounded above (for $\gamma \in (-2,0)$), which implies that the lower bound on the mass and the upper bound on the entropy can be removed from the criteria for smoothness and continuation.  %Our Theorem \ref{t:lower} shows that %Compared to the lower bounds of \cite{HST2017landau}, our Theorem \ref{t:lower} gives more precise information on the behavior as $|v|\to \infty$, but we are working with stronger \emph{a priori} assumptions on $f$ that ensure the right-hand side of \eqref{e:main} is coercive. 
It should also be possible to remove these two assumptions %that $M_f(t,x)$ is bounded below and $H_f(t,x)$ is bounded above 
from the results in the current paper, but we do not explore this here.

%are part of a program to find conditions under which solutions are regular and can be continued past a given time. This began with 

Global existence for the inhomogeneous Landau equation has only been shown for perturbative initial data.  If $f_{in}$ is sufficiently close to Maxwellian equilibrium state $M$, a solution exists for all time and converges to $M$ as $t\to\infty$. This result was found by Guo \cite{guo2002periodic} and later improved by many other authors, see \cite{carrapatoso2016cauchy, mouhot2006equilibrium, strain2008exponential, carrapatoso2017} and the references therein. Recently, Luk \cite{luk2018vacuum} constructed global solutions to \eqref{e:main} for $\gamma \in (-2,0)$ with initial data close to the vacuum state $f\equiv 0$, and showed that these solutions scatter to zero as $t\to\infty$. 
%The spatially homogeneous case, where the solution of \eqref{e:main} is assumed to be independent of $x$, is also better understood than the general case.
Global existence with general initial data remains a challenging open issue (for any value of $\gamma$), but our results imply that a smooth solution exists for as long as the hydrodynamic quantities stay under control as in \eqref{e:hydro}.

In the spatially homogeneous case, where the solution is assumed to be independent of $x$ and the equation has a parabolic structure, much more is known. In this regime, the Landau equation with hard potentials was studied in detail by Desvillettes-Villani \cite{desvillettes2000landau}, who showed global existence and smoothness of solutions with suitable initial data, as well as self-generating polynomial moments and Maxwellian lower bounds. The corresponding study for $\gamma = 0$ was done by Villani \cite{villani1998landau}. Global solutions are also known to exist for $\gamma \in [-2,0)$, see Alexandre-Liao-Lin \cite{alexandre2015apriori} and Wu \cite{Wu2014global}. Generally speaking, when $\gamma \leq 0$, decay of the solution in $v$ (such as higher moments or Gaussian upper bounds) is propagated forward in time, and when $\gamma > 0$, decay is self-generating. A similar dichotomy holds for the inhomogeneous Landau equation, as shown by our Theorem \ref{t:main} and the results of \cite{cameron2017landau}. For other existence and regularity results for the spatially homogeneous case, see \cite{desvillettes2015landau, desvillettes2016general, FG2009Landau, fournier2010coulomb, gualdani2014radial, gualdani2017landau, silvestre2015landau} and the references therein.

%Regarding long-time behavior, it is well understood that solutions of \eqref{e:main} starting sufficiently close to an equilibrium state converge to equilibrium as $t\to \infty$: see 

%For the full inhomogeneous equation, the starting point for conditional regularity 

 %The generating lower bounds of \cite{HST2017landau} have decay proportional to $e^{-\alpha |v|^{2+|\gamma|}}$, which is in fact optimal for soft potentials.

 % we are working with stronger \emph{a priori} assumptions on $f$ that ensure the right-hand side of \eqref{e:main} is coercive. 
%At least in terms of asymptotic behavior for large $|v|$, our Theorem \ref{t:lower} is stronger than the generating lower bounds of \cite{HST2017landau}, which have decay proportional to $e^{-\alpha |v|^{2+|\gamma|}}$. %which is a weaker asymptotic lower bound than our Theorem \ref{t:lower}. 
%The same discrepancy between hard and soft potentials appears in the spatially homogeneous equation (see \cite{desvillettes2000landau}).

Regarding the long-time behavior of large-data solutions to \eqref{e:main}, the famous paper of Desvillettes-Villani \cite{desvillettes2005trend} found that \emph{a priori} global solutions with sufficient smoothness and decay %(uniform in $t$) 
converge almost exponentially to Maxwellians.\footnote{Only the Coulomb case $(\gamma = -3)$ is considered in \cite{desvillettes2005trend}, but one expects that similar techniques can handle other values of $\gamma$, including $\gamma\in (0,1]$.} By the results in the current paper, any global solution satisfying \eqref{e:hydro} with $T=\infty$ satisfies the decay and smoothness hypotheses of \cite{desvillettes2005trend} on $[t_0,\infty)$ for any $t_0>0$.

%with sufficient smoothness and decay will converge to Maxwellians as $t\to \infty$. See, in particular, the famous paper of  for general initial data. 

\subsection{Notation} 
%Equation \eqref{e:main} can be written in divergence form
%\begin{equation}\label{e:divergence}
%\partial_t f + v\cdot \nabla_x f = \nabla_v\cdot \left[\bar a^f(t,x,v)\nabla_v f\right] + \bar b^f(t,x,v)\cdot\nabla_v f + \bar c^f(t,x,v) f,
%\end{equation}
%or in nondivergence form
%\begin{equation}\label{e:nondivergence}
%\partial_t f + v\cdot \nabla_x f = \tr\left[\bar a^f(t,x,v)D_v^2 f\right] + \bar c^f(t,x,v) f,
%\end{equation}
%with the nonlocal coefficients $\bar a^f(t,x,v)\in \R^{d\times d}$, $\bar b^f(t,x,v) \in \R^d$, and $\bar c^f(t,x,v)\in \R$ defined by
%\begin{align}
%\bar a^f(t,x,v) &:= a_{d,\gamma}\int_{\R^d} \left( I - \frac w {|w|} \otimes \frac w {|w|}\right) |w|^{\gamma + 2} f(t,x,v-w) \dd w,\label{e:a}\\
%\bar b^f(t,x,v) &:= b_{d,\gamma}\int_{\R^d} |w|^\gamma w f(t,x,v-w)\dd w,\label{e:b}\\
%\bar c^f(t,x,v) &:= c_{d,\gamma}\int_{\R^d} |w|^\gamma f(t,x,v-w)\dd w, \label{e:c}
%\end{align}
%for some constants $a_{d,\gamma}$, $b_{d,\gamma}$, and $c_{d,\gamma}$. The divergence form of the equation is more convenient for applying local De Giorgi type estimates, and the nondivergence form is more convenient for applying the maximum principle, so we will use both.

Sometimes, we will write $z=(t,x,v) \in [0,T]\times \R^d\times\R^d$. Because of the symmetry properties of the equation, the most convenient sets on which to study local estimates are twisted cylinders of the form
\[Q_r(z_0) = \{z\in \R^{2d+1}: t_0-r^2 < t\leq  t_0, |x-x_0 - (t-t_0)v_0|< r^3, |v-v_0|< r\},\]
for some $z_0 = (t_0,x_0,v_0)$ and $r>0$. We also write $Q_r = Q_r(0,0,0)$. %We say a constant is \emph{universal} if it depends only on $d$, $\gamma$, $m_0$, $M_0$, $E_0$, and $H_0$
We use the notation $A\lesssim B$ when $A\leq CB$ for a universal constant $C$, and $A\approx B$ when $A\lesssim B$ and $B\lesssim A$.

\section{Preliminaries}

In this section, we extend the pointwise upper bounds of \cite{cameron2017landau} to the case $\gamma \in (0,1]$. The proofs are similar to \cite{cameron2017landau}, but we must pay careful attention to the dependence of all constants on $G_0$. We will assume $G_0 \geq 1$, since we want to find an upper bound for $G_f(t,x)$. 
%Our upper bounds are slightly stronger than \cite{cameron2017landau} because we do not require our solutions to be \emph{a priori} bounded. This is because, unlike in the case $\gamma\in (-2,0)$, one needs the 
First, we establish estimates on the coefficients in \eqref{e:main}, relative to our assumptions on $f$: 
\begin{lemma}\label{l:abc}
Let $f$ satisfy \eqref{e:hydro} and \eqref{e:G0}. Then there exist universal constants $c_1$, $C_1$, $C_2$ $C_3$ such that 
\begin{align*}
 \bar a^f_{ij}(t,x,v) e_i e_j &\geq c_1\begin{cases} (1+|v|)^{\gamma}, & e\in \mathbb S^{d-1}, \\
(1+|v|)^{\gamma+2},& e\cdot v = 0,\end{cases}\\
 \bar a^f_{ij}(t,x,v) e_i e_j &\leq G_0 + C_1 \begin{cases} (1+|v|)^{\gamma+2}, &e\in \mathbb S^{d-1},\\
(1+|v|)^{\gamma},  & e\cdot v = |v|,\end{cases}\\
|\bar b^f(t,x,v)| &\leq C_2(1+|v|)^{\gamma+1},\\
\bar c^f(t,x,v) &\leq C_3(1+|v|)^\gamma.
\end{align*} 
\end{lemma}
%(Note that $C_2$ is independent of $\|f\|_{L^\infty}$, unlike in the $\gamma<0$ case.)
\begin{proof} For this proof, the dependence of $f$, $\bar a^f$, $\bar b^f$, and $\bar c^f$ on $t$ and $x$ is irrelevant, so we will write $f(v)$, etc. 

The lower bounds for $\bar a^f_{ij}$ are proven in \cite[Proposition 4]{desvillettes2000landau} in the case $d=3$. The same conclusion for arbitrary dimension $d$ and $\gamma \in [-d,0)$ is established in \cite[Lemma 3.1]{silvestre2015landau}, with an argument that does not use the sign of $\gamma$ in any essential way. As such, we omit the proof of the lower bounds.
	
	For the upper bounds on $\bar a^f_{ij}$, let $e\in \mathbb S^{d-1}$ be arbitrary. From \eqref{e:a}, we have
	\begin{align*}
	\bar a^f_{ij}(v)e_i e_j &\lesssim %\int_{\R^d} |w|^{\gamma+2} f(v-w)\dd w\\
	 \int_{\R^d} |v-z|^{\gamma+2} f(z)\dd z\\
	&\lesssim \int_{\R^d} (|v|^{\gamma+2} + |z|^{\gamma+2})f(z)\dd z\\
	&\lesssim  M_0(1+|v|^{\gamma+2}) +G_0.
	\end{align*}
Next, if $e\in \mathbb S^{d-1}$ is parallel to $v$, then
\begin{align*}
%\int_{\R^d} \left(1 - \left(\frac {w\cdot e}{|w|}\right)^2\right)|w|^{\gamma+2} f(v-w)\dd w 
\bar a^f_{ij}(v) e_i e_j & \lesssim \int_{\R^d} \left(1 - \left(\frac {(v-z)\cdot e}{|v-z|}\right)^2\right)|v-z|^{\gamma+2} f(z)\dd z\\
%&=  \int_{\R^d} \left(|v-z|^2 - \left(|v|-z\cdot e\right)^2\right)|v-z|^{\gamma} f(z)\dd z\\
&=  \int_{\R^d} \left(|z|^2 - (z\cdot e)^2\right)|v-z|^{\gamma} f(z)\dd z\\
%&\leq  \int_{\R^d} |z|^2 |v-z|^{\gamma} f(z)\dd z\\
&\lesssim \int_{\R^d}(|v|^{\gamma}|z|^2+ |z|^{\gamma+2})f(z)\dd z\\
&\lesssim E_0 (1+|v|)^\gamma + G_0.  
\end{align*}

For $\bar b^f$, since $\gamma \leq 1$, we have
\[|\bar b^f(v)| \leq \int_{\R^d} |v-z|^{\gamma+1} f(z)\dd z \lesssim \int_{\R^d}(|v|^{\gamma+1} + |z|^{\gamma+1}) f(z) \dd z \lesssim M_0 (1+|v|)^{\gamma+1} + E_0.\]
The bound on $\bar c^f$ follows by a similar calculation.
\end{proof}

Next, we quote a theorem of \cite{golse2016} that gives a local $L^\infty$ estimate for weak solutions of
	\begin{equation}\label{e:FP}
	\partial_t g + v\cdot \nabla_x g = \nabla_v \cdot(A\nabla_v g) + B \cdot\nabla_v g + s
	\end{equation}
where $A$, $B$, and $s$ are bounded and measurable, and $A$ is uniformly elliptic. When we apply this estimate, the constant will depend on $G_0$ via the upper ellipticity constant $\Lambda$ (see Lemma \ref{l:abc}). To determine the dependence on $\Lambda$, which is not explicitly stated in \cite{golse2016}, we must follow the proof in that article and keep track of the constant $\Lambda$ at every step. This straightforward but tedious task is outlined in Appendix \ref{s:A}. 
%In order to keep track of the constant $G_0$, we need to estimate how the constant in the following theorem depends on the upper ellipticity constant $\Lambda$. 
%This dependence is not explicitly stated in \cite{golse2016}, but it is straightforward to verify that the constant is bounded by some polynomial factor of $\Lambda$, which is precise enough for our purposes.
\begin{theorem}\label{t:local}
	Let $g$ be a weak solution to \eqref{e:FP} in $Q_1$. Then there holds
	\begin{equation}\label{e:degiorgi}
	\|g\|_{L^\infty(Q_{1/2})} \leq C \left(\|g\|_{L^2(Q_1)} + \|s\|_{L^\infty(Q_1)}\right),
	\end{equation}
	where $C = C'(1+\Lambda)^P$ for some $P>0$ depending only on the dimension $d$, and $C'$ depending only on $d$ and  $\lambda$.
\end{theorem}
\begin{proof}
	See \cite[Theorem 12]{golse2016}. The form of $C$ is justified in Appendix \ref{s:A} below.
\end{proof}

As in \cite{cameron2017landau}, we can apply scaling techniques to \eqref{e:degiorgi} and derive an improved pointwise estimate:
\begin{proposition}\label{p:improved}
	Let $g$ be a weak solution of \eqref{e:FP} in $Q_R$ for some $R>0$, with
	\begin{align*}
	0 < \lambda I \leq A(t,x,v) \leq \Lambda I, \qquad &(t,x,v)\in Q_R,\\
	|B(t,x,v)| \leq \Lambda/R,\qquad &(t,x,v)\in Q_R,\\
	s\in L^\infty(Q_R). \qquad &
	\end{align*}
	%and assume that $R^{-d}\|g\|_{L_{t,x}^\infty L_v^1(Q_R)} \leq \|s\|_{L^\infty(Q_R)}$. 
	Then the estimate
	\begin{equation*}
	g(0,0,0) \leq C\left(\|g\|_{L_{t,x}^\infty L_v^1(Q_R)}^{2/(d+2)} \|s\|_{L^\infty(Q_R)}^{d/(d+2)} + R^{-d}\|g\|_{L_{t,x}^\infty L_v^1(Q_R)}\right)
	\end{equation*}
	holds, with $C=C'(1+\Lambda)^P$ as in Theorem \ref{t:local}.
\end{proposition}
\begin{proof}
	See \cite{cameron2017landau}, Proposition 3.2 and Lemma 3.3.
\end{proof}

To apply Proposition \ref{p:improved} to the Landau equation, since the ellipticity constants of $\bar a^f_{ij}$ degenerate as $|v|\to \infty$, we need a change of variables in $v$ that produces an equation with universal ellipticity constants in a small cylinder. This change of variables was first introduced in \cite{cameron2017landau} for the case $\gamma \in (-2,0)$.

\begin{lemma}\label{l:T}
	Let $z_0 =(t_0,x_0,v_0)\in \R_+\times \R^{d}\times \R^d$ be such that $|v_0|\geq 2$, let $S$ be the linear transformation such that
	\begin{equation*}
	S e = \begin{cases}   |v_0|^{1+\gamma/2} e , & e \cdot v_0 = 0\\
	|v_0|^{\gamma/2}e, & e \cdot v_0 = |v_0|,\end{cases}
	\end{equation*}
	and define
	\begin{align*}
	\mathcal S_{z_0}(t,x,v) := (t_0+t,x_0+S x + t v_0 ,v_0 + Sv).
	\end{align*}
	Then there exists a radius 
	\[r_1 = c_1 |v_0|^{-1-\gamma/2}\min\left(1,\sqrt{t_0/2}\right),\] 
	with $c_1$ universal, such that:
	\begin{enumerate}
		\item[\textup{(a)}] There exists a constant $C>0$ independent of $v_0\in\R^d\setminus B_2$ such that for all $v\in B_{r_1}$,
		\[ C^{-1} |v_0| \leq |v_0 + Sv| \leq C |v_0|.\]
		\item[\textup{(b)}] If $f$ is a weak solution of the Landau equation \eqref{e:main} satisfying \eqref{e:hydro} and \eqref{e:G0}, then $f_{z_0}(t,x,v) := f(\mathcal S_{z_0}(t,x,v))$ satisfies
		\begin{equation}\label{e:isotropic}
		\partial_t f_{z_0} + v \cdot \nabla_x f_{z_0} = \nabla_v \cdot\left(A(z)\nabla_v f_{z_0}\right) +   B(z)\cdot \nabla_v f_{z_0} +   C(z) f_{z_0}
		\end{equation}
		in $Q_1$, and the coefficients 
		\[\begin{split}
		&A(z) = S^{-1}\bar a^f(\mathcal S_{z_0}(z)) S^{-1},
		\quad B(z) = S^{-1}\bar b^f(\mathcal S_{z_0}),\quad C(z) = \bar c^f(\mathcal S_{z_0}(z))
		\end{split}\]
		satisfy
		\begin{align*}
		\lambda I &\leq  A(z) \leq \Lambda I,\\
		| B(z)| &\lesssim |v_0|^{1+\gamma/2},\\
		| C(v)| &\lesssim |v_0|^{\gamma},
		\end{align*}
		with $\Lambda \lesssim G_0^P$, and $\lambda$ and the bounds on $B(z)$ and $C(z)$ depending only on universal constants. 
	\end{enumerate}
\end{lemma}
\begin{proof}
	To prove (a), since $v\in B_{r_1}$, we have
	\[ |Sv| \leq c_1(1+|v_0|)^{1+\gamma/2}|v| \leq c_1,\]
	so that $|v_0+Sv| \approx |v_0|$ for $|v_0| \geq 2$.
	
	For (b), the equation \eqref{e:isotropic} satisfied by $f_{z_0}$ follows by direct computation. The uniform ellipticity of $A(z)$ (with constants independent of $|v_0|$) is the only subtle part of this lemma, and is the reason we must take $|v|\lesssim |v_0|^{-1-\gamma/2}$. The proof can be found in \cite[Lemma 4.1]{cameron2017landau}. For the bound on $B(z)$, Lemma \ref{l:abc} and conclusion (a) imply
	\begin{align*} 
	|B(z)| &\lesssim \|S^{-1}\| |\bar b^f(\mathcal S_{z_0}(z))| \lesssim (1+|v_0|)^{1+\gamma/2}.
	\end{align*}
	The bound on $C(z)$ follows similarly.
\end{proof}

Next, we find global upper bounds for any solution, that depend only on universal constants and $G_0$. The proof also gives some polynomial decay, but we will not make any use of this.
\begin{theorem}\label{t:upper}
	Let $f:[0,T]\times \R^{2d}\to\R_+$ be a bounded weak solution of \eqref{e:main} satisfying \eqref{e:hydro} and \eqref{e:G0}. Then 
	\[f(t,x,v) \leq C \left(1+t^{-d/2}\right)(1+|v|)^{-1-\gamma},\]
	with $C \lesssim G_0^P$ for some $P$ depending on dimension. 
\end{theorem}
The $P$ in this theorem is not the same as the $P$ in Theorem \ref{t:local}.
\begin{proof}
 	Define
 	\[K:= \sup\limits_{(0,T]\times \R^{2d}} \min\{t^{d/2},1\}f(t,x,v). \]
	We may assume $K \geq 1$. 	Let $\eps>0$, and choose $z_0 = (t_0,x_0,v_0)\in (0,T]\times \R^{2d}$ such that $f(t_0,x_0,v_0)> (K-\eps)\max\{t_0^{-d/2}, 1\}$. Define $r_0 = \min\{1,\sqrt{t_0}\}/2$, and note that $r_0^{-d} \approx (1+t_0^{-d/2})$.
	
	If $|v_0|\leq 2$, Lemma \ref{l:abc} and Proposition \ref{p:improved} with $g(z) = f(z_0+z)$ and $s(z) = \bar c^f(z_0+z) f(z_0+z)$ in $Q_{r_0}(z_0)$ imply 
	\begin{equation}\label{e:small-v}
	f(t_0,x_0,v_0) \leq C G_0^P \left(M_0^{2/(d+2)} K^{d/(d+2)} + r_0^{-d}M_0\right) \lesssim G_0^P r_0^d K^{d/(d+2)}.
	\end{equation}
	
	If $|v_0|\geq 2$, let $r_1 = c_1 r_0 |v_0|^{-1-\gamma/2}$ and $f_{z_0}$ be as in Lemma \ref{l:T}. %, let %We will estimate $f_{z_0}(t,x,v)$ in $Q_R$, where
%	\[R := c_1(r_0/2)(1+|v_0|)^{-1-\gamma/2},\]
%	with $c_1$ as in Lemma \ref{l:T}. 
	By Lemma \ref{l:T}(b), $f_{z_0}$ solves \eqref{e:isotropic} in $Q_{r_1}$, and 
	\begin{align}
	0<&\lambda I \leq  A(z) \leq \Lambda I,\nonumber\\
	|B(z)|&\lesssim |v_0|^{1+\gamma/2},\label{e:Bbound}\\
	|C(z)| &\lesssim  |v_0|^{\gamma},\nonumber
	\end{align}
	in $Q_{r_1}$, where $\Lambda \lesssim G_0^P$ and all other constants are universal.	By Lemma \ref{l:T}(a) and the definition of $K$,
	 \begin{equation}\label{e:QRbound}
	 \|C f_{z_0}\|_{L^\infty(Q_{r_1})} \lesssim K r_0^{-d} |v_0|^\gamma.
	 \end{equation}
	Let $Q_{S,r_1}$ be the image of $Q_{r_1}$ under $z \mapsto \mathcal S_{z_0}(z)$, and note that
	\begin{align}
	\| f_{z_0} \|_{L_{t,x}^\infty L_v^1(Q_{r_1})} &= \det(S^{-1})\|f\|_{L_{t,x}^\infty L_v^1(Q_{S,r_1})}\nonumber\\
	& = |v_0|^{-[(d-1)(2+\gamma)/2 + \gamma/2]}\|f\|_{L_{t,x}^\infty L_v^1(Q_{S,r_1})}\nonumber\\
	&\leq |v_0|^{-\left(1+\gamma+d(2+\gamma)/2\right)}G_0,\label{e:L1bound}
	\end{align}
	where the last inequality comes from \eqref{e:G0} and Lemma \ref{l:T}(a). %the fact that $|v|\gtrsim |v_0|$ if $(t,x,v)\in Q_{T,R}$. 	
	By \eqref{e:Bbound}, we can apply Proposition \ref{p:improved} in $Q_{r_1}$ with $g = f_{z_0}$ and $s = C(z) f_{z_0}$ to obtain
	\begin{align}\label{e:estimate-small-gamma}
	f(t_0,x_0,v_0) &\leq C\left(\|f_{z_0}\|_{L_{t,x}^\infty L_v^1(Q_{r_1})}^{2/(d+2)} \| C(z) f_{z_0}\|_{L^\infty(Q_{r_1})}^{d/(d+2)}+ r_0^{-d}|v_0|^{d(2+\gamma)/2}\|f_{z_0}\|_{L_{t,x}^\infty L_v^1(Q_{r_1})}\right)\nonumber\\
	&\leq C'G_0^P \left(G_0^{2/(d+2)}(Kr_0^{-d})^{d/(d+2)} |v_0|^{-1-(d+2\gamma)/(d+2)} + r_0^{-d}G_0 |v_0|^{-1-\gamma}\right),\nonumber\\
	&\leq C' G_0^{P+1} K^{d/(d+2)} r_0^{-d} |v_0|^{-1-\gamma},
	\end{align}
	using \eqref{e:QRbound} and \eqref{e:L1bound}.
	
	By our choice of $(t_0,x_0,v_0)$, \eqref{e:small-v} and \eqref{e:estimate-small-gamma} imply
	\[(K - \eps)r_0^{-d} \leq C' G_0^{P+1} K^{d/(d+2)}r_0^{-d},\]
	and since this is true for any $\varepsilon >0$, we have $K \leq C G_0^{(P+1)(d+2)/2}$. Applying \eqref{e:estimate-small-gamma} again, we conclude $f(t_0,x_0,v_0) \leq C G_0^{(P+1)d/2} r_0^{-d} (1+|v_0|)^{-1-\gamma}$.
\end{proof}

\section{Gaussian decay}

In this section, we show that all bounded solutions have Gaussian decay. The proof relies on the maximum principle for $H^1$ weak solutions of the linear Landau equation, which can be found in, e.g., the appendix of \cite{cameron2017landau}. Given a function $f$ satisfying \eqref{e:hydro} and \eqref{e:G0}, a solution $\phi$ to the linear Landau equation satisfies
\[\partial_t \phi + v\cdot \nabla_x \phi = \nabla_v\cdot(\bar a^f \nabla_v\phi) + \bar b^f \cdot \nabla_v \phi + \bar c^f \phi.\]
Note that if $\phi$ is smooth, this is equivalent to
\[ \partial_t \phi + v\cdot \nabla_x \phi = \tr(\bar a^f D_v^2 \phi) + \bar c^f \phi,\]
since $\bar b^f_i = -\partial_{v_j} \bar a^f_{ij}$. %We are allowed to differentiate $\bar a^f$ in $v$, since the convolution kernel is differentiable.

First, we show Gaussians with appropriate decay constants are super- or sub-solutions of the linear Landau equation for large velocities:
\begin{lemma}\label{l:super}
Let $f$ be a bounded function satisfying \eqref{e:hydro} and \eqref{e:G0}, and let $\bar a^f$ and $\bar c^f$ be defined by \eqref{e:a} and \eqref{e:c} respectively.
	\begin{enumerate}
	\item[(a)]
	 There exists a universal constant $C$ such that, if $\alpha = \dfrac C {G_0}$, then the function 
	\[\phi(v) := e^{-\alpha |v|^2}\]
	satisfies
	\begin{equation*}%\label{e:super}
	\bar a^f_{ij} \partial_{ij} \phi + \bar c^f \phi \leq -C G_0^{-1}|v|^{\gamma+2}\phi,
	\end{equation*}
	for $|v|\geq CG_0^{1/2}$.
	\item[(b)] There exists $C$ universal such that, if $\alpha =  C G_0$, then $\phi(v)$ defined as in (a) satisfies 
	\[\bar a^f_{ij}\partial_{ij} \phi + \bar c^f \phi \geq C G_0^2 |v|^{\gamma+2} \phi,\]
	for $|v|\geq CG_0^{-1/2}$.
	\end{enumerate}
\end{lemma}
\begin{proof}
	Let $\alpha$ be a positive constant to be chosen later. Since $\phi$ is radial, for any $v\neq 0$ we have
	\begin{equation*}%\label{e:D2phi}
	\partial_{ij}\phi = \frac {\partial_{rr}\phi}{|v|^2}   v_i v_j + \frac{\partial_r\phi}{|v|} \left( \delta_{ij} - \frac{v_iv_j}{|v|^2}\right) = \left[\frac {4\alpha^2|v|^2 - 2\alpha}{|v|^2}   v_i v_j -  2\alpha  \left( \delta_{ij} - \frac{v_iv_j}{|v|^2}\right)\right] e^{-\alpha |v|^2},
	\end{equation*}
	and Lemma \ref{l:abc} implies
	\begin{align*}
	\bar a^f_{ij}\partial_{ij}\phi &\leq \left[(4\alpha^2|v|^2 - 2\alpha)C_1G_0 |v|^{\gamma}  -  2\alpha c_1 |v|^{\gamma+2}\right] e^{-\alpha |v|^2}\\
	&= \left[(4\alpha^2 C_1G_0 - 2\alpha c_1) |v|^{\gamma+2} - 2\alpha C_1G_0 |v|^{\gamma}\right] e^{-\alpha|v|^2},
	\end{align*}
	so for $\alpha = c_1 / (4C_1G_0)$, we have $\bar a^f_{ij}\partial_{ij}\phi \leq  -C G_0^{-1}|v|^{\gamma+2}$. With the bound on $\bar c^f$ in Lemma \ref{l:abc}, this implies
	\[\bar a^f_{ij} \partial_{ij}\phi  + \bar c^f \phi \leq \left[-CG_0^{-1}|v|^{\gamma+2} + C|v|^{\gamma}\right]\phi(v).\]
	This right-hand side is bounded by $-CG_0^{-1}|v|^{\gamma+2}$ if 
	\[|v|\geq CG_0^{1/2},\]
	for some (new) universal constant $C$, which establishes (a).
	
	For (b), the upper and lower bounds in Lemma \ref{l:abc} imply
	\begin{align*}
	\bar a^f_{ij}\partial_{ij}\phi &\geq \left[(4\alpha^2|v|^2 - 2\alpha)c_1 |v|^{\gamma}  -  2\alpha C_1G_0 |v|^{\gamma+2}\right] e^{-\alpha |v|^2}\\
	&= \left[(4\alpha^2 c_1 - 2\alpha C_1G_0) |v|^{\gamma+2} - 2\alpha c_1 |v|^\gamma\right] e^{-\alpha|v|^2},
	\end{align*}
so that if $\alpha = C_1 G_0 / (4c_1)$, we have $\bar a^f_{ij} \partial_{ij} \phi \geq C G_0^2 |v|^{\gamma+2} - CG_0 |v|^\gamma \geq C G_0^2 |v|^{\gamma+2}$, provided $|v| \geq CG_0^{-1/2}$. Since $\bar c^f\phi \geq 0$, we are done.
\end{proof}

We are ready to prove the first assertion of Theorem \ref{t:main}.
\begin{theorem}\label{t:gaussian}
%\begin{enumerate}
%	\item[(a)] If $f_{in} \leq C_1 e^{-\alpha |v|^2}$ with $\alpha$ as in Lemma \ref{l:super}, then 
%	   \[ f(t,x,v) \leq C e^{-\alpha|v|^2}\]
%	   for all $(t,x,v)$, where $C$ depends on unvi
For any weak solution $f$ of \eqref{e:main} satisfying \eqref{e:hydro} and \eqref{e:G0}, we have
\[f(t,x,v) \leq C G_0^P\left(e^{Ct^{-2/\gamma}}+1\right)e^{-\alpha |v|^2}\]
for all $t>0$, with $P$ as in Theorem \ref{t:upper}, $\alpha$ from Lemma \ref{l:super}(a), and $C$ universal.
\end{theorem}

\begin{proof} 

Let $p = 2/\gamma$ and $\psi(t,x,v) = e^{C_0t^{-p}} e^{-\alpha |v|^2}$, with $C_0$ to be determined. For $|v| \geq C G_0^{1/2}$,  Lemma \ref{l:super} implies
\[-\partial_t \psi - v\cdot \nabla_x \psi+ \bar a^f_{ij}\partial_{ij}\psi + \bar c^f \psi \leq \psi\left(C_0 p t^{-p-1} - CG_0^{-1}|v|^{\gamma+2}\right),\]
which is nonpositive whenever 
\[|v|\geq \max\{C G_0^{1/2}, (p C_0 G_0 t^{-p-1}/C)^{1/(\gamma+2)}\}\geq C C_0^{1/(\gamma+2)} G_0^{1/2} t^{-(p+1)/(\gamma+2)}.\] 

On the other hand, let $|v|^2 \leq C C_0^{2/(\gamma+2)} G_0t^{-2(p+1)/(\gamma+2)}$. Since $p = 2/\gamma$, we have $2(p+1)/(\gamma+2) = p$. Therefore, since $\alpha = CG_0^{-1}$, we have
\[ \psi(t,x,v) = \exp(C_0 t^{-p}-\alpha|v|^2) \geq  \exp(C_0 t^{-p} -  C C_0^{2/(\gamma+2)}t^{-p}),\]
so for $C_0$ sufficiently large (depending only on universal constants), this right-hand side approaches $\infty$ as $t\to 0$.

Let $C_1 \lesssim G_0^P$ be the constant from Theorem \ref{t:upper}. For $t\in (0,1]$, we have $2C_1\psi(t,x,v) \geq 2C_1 e^{C_0 t^{-p}}\geq C_1(1+t^{-d/2}) \geq f(t,x,v)$ in 
\[\{(t,x,v): 0<t\leq 1, |v|^2 \leq C C_0^{2/(\gamma+2)} G_0 t^{-p} \}.\] 
Therefore, we have
\[(-\partial_t - v\cdot \nabla_x + \bar a^f_{ij}\partial_{ij} + \bar c^f)[f - 2C_1\psi]_+ \geq 0, \quad \mbox{in } (0, 1] \times \R^{2d},\]
and $[f(0,x,v)-2C_1\psi(0,x,v)]_+ = 0$ for all $(x,v)\in \R^{2d}$. The maximum principle implies $[f-2C_1\psi]_+ = 0$ in $(0, 1] \times \R^{2d}$, so that
\[f(t,x,v) \leq 2C_1 e^{C_0 t^{-p}}e^{-\alpha|v|^2}, \quad t\in (0,1].\]

Next, with $R_0 = CG_0^{1/2}$ as in Lemma \ref{l:super}(a), choose $C_2$ such that $C_2e^{-\alpha R_0^2}\geq 2C_1 \geq  \|f\|_{L^\infty([1,T]\times \R^{2d})}$. Since $\alpha R_0^2$ is bounded independently of $G_0$, we can choose $C_2 \lesssim C_1$. Then we can apply the maximum principle to $[f-(e^{C_0}+C_2) e^{-\alpha |v|^2}]_+$ on $[1,T]\times \R^{2d}$ and conclude the proof.
\end{proof}

We now show that the $(\gamma+2)$ moment of $f$ is bounded independently of $G_0$ on $[t_0,T]$ for any $t_0>0$. (We are seeking a bound that does not depend quantitatively on the $(\gamma+2)$ moment of the initial data, so we cannot hope for a bound that is uniform in $t\in [0,T]$.)

\begin{theorem}\label{t:gamma2}
	With $f$ as above, for any $\eps>0$, there exists a constant $C_\eps$, depending only on universal constants and $\eps$, such that 
	\[\int_{\R^d} |v|^{\gamma+2} f(t,x,v) \dd v \leq C_\eps \left(1+t^{-1/(1-\gamma/2 - \eps)}\right), \quad t\in (0,T], \,x\in \R^d.\]
%	where $C$ depends only on universal constants and $t_0$.
\end{theorem}

%In particular, we have $\int_0^{T} \sup_{x\in\R^d} G_f(t,x) \dd t < \infty$. (Compare \cite[Theorem 3]{desvillettes2000landau} in the spatially homogeneous case.) 
\begin{proof}
Let $t_* \in (0,1]$ be arbitrary. By Theorem \ref{t:gaussian}, $f$ is bounded by $K e^{-\alpha |v|^2}$ on $[t_*,T]$, with $\alpha = C/G_0$ and $K \lesssim e^{Ct_*^{-2/\gamma}} G_0^P$. We will interpolate between this pointwise Gaussian decay and the energy bound. 
%Suppose $f \leq C \Lambda^{Q} e^{-\alpha |v|^2}$ with $\alpha \leq C(1+G_0)^{-1}$ and $\Lambda \le%q C(1+G_0)$. Here, $Q$ depends only on the dimension $d$. 
For $p>1$ to be chosen later, let $q$ be such that $1/p + 1/q = 1$. For $t\geq t_*$, we have
\begin{align*}
\int_{\R^d} |v|^{\gamma+2} f(t,x,v) \dd v &\leq K^{1/p} \int_{\R^d} e^{-\alpha|v|^2/p} f^{1/q} |v|^{\gamma+2/p} |v|^{2/q} \dd v\\
&\leq K^{1/p} \left( \int_{\R^d} e^{-\alpha |v|^2} |v|^{p\gamma + 2} \dd v\right)^{1/p}\left(\int_{\R^d} |v|^2 f\dd v\right)^{1/q}\\
&\leq K^{1/p} E_0^{1/q} C_p \alpha^{-(d+2)/(2p) - \gamma/2} \\
&\leq  e^{(C t_*^{-2/\gamma})/p}G_0^{P/p} E_0 C_p (C G_0)^{(d+2)/(2p) + \gamma/2},
\end{align*} 
with 
\begin{align*}
C_p = \left(\int_{\R^d} e^{-|w|^2} |w|^{p\gamma+2} \dd w\right)^{1/p} &= \left(C_d\int_0^\infty e^{-r^2} r^{p\gamma+d+1} \dd r\right)^{1/p} %\\ 
%&= \left(C_d\int_0^\infty e^{-s} s^{(p\gamma+d)/2} \dd s\right)^{1/p}
 = \left(C_d \Gamma\left(\frac {p\gamma + d }{2}+1\right)\right)^{1/p}.
 \end{align*}
%and $C_d$ depending on dimension. 
It follows from Stirling's approximation that $\lim_{x\to \infty} \Gamma(ax+b)^{1/x}/x^a$ exists for any $a, b >0$. % (which follows from Stirling's approximation),
Therefore, we have $C_p \lesssim p^{\gamma/2}$. Choosing %Let $\eps \in (0,1/2)$, and choose 
\[p = \max\{Ct_*^{-2/\gamma}, [(d+2)/2 + P]/\eps\},\] 
we obtain
\[\int_{\R^d} |v|^{\gamma+2} f(t,x,v) \dd v \leq  C e^1 E_0 G_0^{\gamma/2+\eps}p ^{\gamma/2} \leq C_\eps t_*^{-1} G_0^{\gamma/2+\eps},\]
with $C_\eps$ depending on universal constants and $\eps$. % Choose $\eps>0$ such that $\gamma + \eps <1$, and choose $p$ larger if necessary so that $(d+2)/2p < \eps$. (We can assume $p \approx G_0t_0^{-2/\gamma}$ still, because otherwise the $\gamma+2$ moment is bounded above depending only on $d$, $\gamma$, and $t_0$.) 
Let $\mu = \gamma/2+\eps < 1$. Then we finally have
\begin{equation}\label{e:t0}
\int_{\R^d} |v|^{\gamma+2} f(t,x,v) \dd v \leq C_\eps t_*^{-1}G_0^{\mu}, \quad t\geq t_*.
\end{equation}

To apply \eqref{e:t0} iteratively, we need to wait a short amount of time at each step. Let $t_0 \in(0,1]$ and $n\in \mathbb N$ be arbitrary, and define the following sequence of times:
\[t_{1,n} = 2^{-n+1}t_0, \,\, t_{2,n} = 2^{-n+2}t_0, \,\, t_{3,n} = 2^{-n+3}t_0, \ldots,t_{n-1,n} = t_0/2,\,\, t_{n,n} = t_0.\]
Apply \eqref{e:t0} to $f$ with $t_* = t_{1,n}$:
\[\int_{\R^d} |v|^{\gamma+2} f(t,x,v) \dd v \leq  C_\eps t_0^{-1} 2^{n-1} G_0^{\mu} =: G_{1,n}, \quad t\geq t_{1,n}.\]
Using the new bound $G_{1,n}$ for $G_f(t,x)$, we apply \eqref{e:t0} to $f(t_{1,n}+t,x,v)$ with $t_* = t_{2,n}-t_{1,n} = 2^{-n+1}t_0$ and obtain
\[\int_{\R^d} |v|^{\gamma+2} f(t,x,v) \dd v \leq C_\eps t_0^{-1} 2^{n-1} G_{1,n}^{\mu} = (C_\eps t_0^{-1})^{1+\mu} 2^{(n-1)(1+\mu)} G_0^{\mu^2} =: G_{2,n}, \quad t\geq t_{2,n}.\]
Continuing, we apply \eqref{e:t0} to $f(t_{2,n} + t,x,v)$ with $t_* = t_{3,n} - t_{2,n} = 2^{-n+2}t_0$:
\begin{align*}
\int_{\R^d} |v|^{\gamma+2} f(t,x,v) \dd v &\leq C_\eps t_0^{-1} 2^{n-2} G_{2,n}^{\mu}\\
 &= (C_\eps t_0^{-1})^{1+\mu+\mu^2} 2^{n-2} 2^{(n-1)(\mu+\mu^2)} G_0^{\mu^3} =: G_{3,n}, \quad t\geq t_{3,n}.
\end{align*}
Repeating this $n$ times, we have
\[\int_{\R^d} |v|^{\gamma+2} f(t,x,v) \dd v \leq  (C_\eps t_0^{-1})^{1+\mu+\cdots+\mu^{n-1}} C_n G_0^{\mu^n}   =:G_{n,n}, \quad t\geq t_0,\]
where 
\[C_n = 2\cdot 2^{2\mu} \cdot 2^{3\mu^2}\cdots 2^{(n-2)\mu^{n-3}} 2^{(n-1)[\mu^{n-2}+\mu^{n-1}]}.\]
Let the number of steps $n\to \infty$, and we have
\[C_n \to 2^{\sum_{k=1}^\infty k\mu^{k-1}} = 2^{1/(1-\mu)^2}. \]
This implies
\[\int_{\R^d} |v|^{\gamma+2} f(t,x,v) \dd v \leq  C(C_\eps t_0^{-1})^{1/(1-\mu)}, \quad t\geq t_0,\]
as desired. For general $t_0 \in (0,T]$, we proceed as above, replacing $t_0$ with $\min\{1,t_0\}$, and conclude the statement of the theorem.
\end{proof}

We can now prove the second assertion of Theorem \ref{t:main}: %For any $t_0 >0$, $G_f(t,x)$ is bounded independently of $G_0$ on $[t_0/2,T]\times \R^{2d}$ by 
Theorem \ref{t:gamma2} implies $G_f(t,x) \leq C_\eps (1+(t_0/2)^{-1/(1-\gamma/2-\eps)})$ on $[t_0/2,T]\times \R^{2d}$, for some $\eps\in (0,1/2)$. Applying Theorem \ref{t:gaussian} to $f(t_0/2+t,x,v)$ with this new bound for $G_f(t,x)$, we conclude 
\[f(t_0,x,v)\leq C (t_0/2)^{-P/(1-\gamma/2-\eps)} e^{C(t_0/2)^{-2/\gamma}} e^{-C t_0^{1/(1-\gamma/2-\eps)}|v|^2}, \]
as desired. 
%with $J(t_0)$ and $\beta(t_0)$ depending only on universal constants and $t_0$. %(In particular, they are independent of $G_0$.)

%(Do we need to do the maximum principle again with time-dependent $\phi$?)

%(Optimize this proof to get better exponent of $E_0$? For now, $C$ also depends on $E_0$, but possibly not in Gaussian lower bound case.)

Finally, we derive Gaussian lower bounds for $f$. %at $t=0$ are propagated for positive times. In this case, we cannot get a useful statement that is independent of $G_0$, since the assumption on the initial data degenerates as the bound on $G_f(t,x)$ goes to $\infty$. 
%The proof of part (b) of the following theorem is reminiscent of \cite[Theorem 9]{desvillettes2000landau}.

\begin{proof}[Proof of Theorem \ref{t:lower}] Throughout this proof, we will write 
	\[L = -\partial_t - v\cdot \nabla_x + \bar a^f_{ij}\partial_{v_iv_j} + \bar c^f.\]
As usual, we sum over repeated indices.

(a) Let $\psi_1(t,x,v) = e^{-\mu t} e^{-\alpha |v|^2}$, with $\mu>0$ to be determined. For $|v|\geq R_0 := C G_0^{-1/2}$, Lemma \ref{l:super}(b) implies 
\[L\psi_1 \geq \left(\mu + CG_0^2 |v|^{\gamma+2}\right)\psi_1 \geq 0. \]
(Note that $\bar c^f \psi_1 \geq 0$.) For $|v|\leq R_0$, it follows from Lemma \ref{l:abc} that 
\[\bar a^f_{ij} \partial_{v_iv_j}\psi_1 \geq -CG_0(1+|v|)^{2+\gamma}|D_v^2 \psi_1| \geq -CG_0 R_0^{\gamma+4} |\psi_1|,\]
 which is bounded from below by some constant depending on $R_0$. Choosing $\mu$ sufficiently large, we have $L\psi_1 \geq 0$ in $(0,T]\times \R^{2d}$, and the conclusion follows from applying the maximum principle to $c\psi_1 - f$.

(b) %By assumption, we have $f(0,x,v)\geq \delta$ in $B_r(v_0)$, for all $x\in \R^d$. We may suppose $r\leq 1$. 
For any $t_1\in (0,T]$ and $x\in \R^d$, the hydrodynamic bounds \eqref{e:hydro} imply for any $r>0$,
\[m_0 \leq \int_{\R^d} f(t_1,x,v) \dd v \leq \|f(t_1,x,\cdot)\|_{L^\infty(B_r)} C_d r^d + E_0 r^{-2}.\] 
Clearly, there exists $r>0$ such that $r = \|f(t_1,x,\cdot)\|_{L^\infty(B_r)}^{-1/(d+2)} E_0^{1/(d+2)}$. Theorem \ref{t:upper} implies $r\gtrsim t_1^{d/(2(d+2))}$. With this choice of $r$, we have $\|f(t_1,x,\cdot)\|_{L^\infty(B_r)} \gtrsim m_0^{(d+2)/2} E_0^{d/2}$.  Applying a scaled version of the Harnack inequality \cite[Theorem 4]{golse2016}, we have for any $t_2 > t_1$,
\begin{equation}\label{e:Br}
\inf_{v\in B_r} f(t_2,x,v) \geq \delta, \quad \mbox{for all } x\in \R^d,
\end{equation}
with $\delta>0$ depending on $t_2-t_1$, $r$, and universal constants. 

Now, let $t_0>0$ be arbitrary, and apply \eqref{e:Br} with $t_1 = t_0/2$ and $t_2 = t_0$. The constants $r$ and $\delta$ depend only on universal constants and $t_0$, and they are nondecreasing as $t_0$ increases. (This can be seen either by tracking the dependencies in \eqref{e:Br} or by applying the same reasoning to $f(t'+t,x,v)$ for any $t'>0$.) We conclude
\begin{equation}\label{e:Br2}
\inf_{v\in B_r} f(t,x,v) \geq \delta, \quad \mbox{for all } t\geq t_0, \,x\in \R^d,
\end{equation}
with $\delta, r>0$ depending on universal constants and $t_0$.

Next, we show $f(t,x,\cdot)$ is bounded below by a Gaussian. %Define $g(t,x,v) = f(t,x + tv_0,v_0+v)$, and note that $Lg = 0$. Let $G(t) = C_\eps (1+t^{-1/(1-\gamma/2-\eps)})$ be the upper bound for $G_f(t,x)$ provided by Theorem \ref{t:gamma2}, with some arbitrary choice of $\eps\in (0,1/2)$. 
Define 
%\[\psi_3(t,x,v) = e^{-\mu (t-t_0)} e^{- \rho[1+ (t-t_0)^{-1}]|v|^2},\] 
\[\psi_2(t,x,v) = e^{- \rho[1+ (t-t_0)^{-1}]|v|^2},\] 
with $\rho>0$ to be determined. Letting $\Omega = \{t\geq t_0, x\in \R^d, |v| \geq r/2\}$, the function $\psi_2$ can be extended smoothly by $0$ on the part of $\partial \Omega$ with $t= t_0$ (since $|v| \geq r/2$ in $\overline \Omega$). 
By \eqref{e:Br2}, we have $f\geq \delta \psi_2$ on $\partial\Omega$. It remains to show $\psi_2$ is a subsolution in $\Omega$. Let $\omega(t) =  \rho [1+(t-t_0)^{-1}]$. By the calculations of Lemma \ref{l:super}(b), we have for $|v|\geq r/2$,
\[\bar a^f_{ij} \partial_{v_iv_j} \psi_2 \geq \left[\left(4\omega^2(t) c_1 - 2C(t_0)\omega(t)\right) |v|^{\gamma+2} - 2\omega(t) c_1 |v|^\gamma\right] \psi_2, \]%\geq [4\omega^2(t) r^{2+\gamma}c_1 - C \omega(t) (1+r^{\gamma})]\psi,\]
where $C(t_0)$ is given by Theorem \ref{t:gamma2}. For $\rho$ large enough (depending on universal constants, $t_0$, and $r$) we have $\bar a^f_{ij} \partial_{v_iv_j} \psi_2 \geq C \omega^2(t) |v|^{\gamma+2}\psi_2$ in $\Omega$. This implies
\begin{align*}
L\psi_2 &\geq \left[\omega'(t)|v|^2 + C \omega^2(t) |v|^{\gamma+2} \right] \psi_2.
\end{align*}
Since $\omega'(t) = -\rho (t-t_0)^2$, we can choose $\rho$ sufficiently large, depending on universal constants, $t_0$, and $r$, such that $\omega^2(t)|v|^{\gamma+2} \geq \omega'(t)|v|^2$ and $L\psi_2 \geq 0$ for $|v|\geq r/2$. % \geq Since $G^2(t) = C t^{-2/(1-\gamma/2-\eps)}$ approaches $+\infty$ faster than $G'(t) = C t^{-(2-\gamma/2-\eps)/(1-\gamma/2-\eps)}$ as $t\to 0$, our expression for $L\psi_3$ is nonnegative for $|v|\geq r/2$ 
Applying the maximum principle in $\Omega$, we conclude 
%\begin{equation}\label{e:lower_bound}
%f(t,x,v) \geq \delta e^{-\mu(t-t_0)} e^{-\rho[1+(t-t_0)^{-1}]|v|^2}, \quad t\geq t_0.
%\end{equation}
\begin{equation}\label{e:lower_bound}f(t,x,v) \geq \delta e^{-\rho[1+(t-t_0)^{-1}]|v|^2}, \quad t\geq t_0.
\end{equation}
The constants $\delta = \delta_{t_0}$ and $\rho = \rho_{t_0}$ degenerate as $t_0\to 0+$. Replacing $t_0$ with $t_0/2$ in \eqref{e:lower_bound}, we conclude $f(t,x,v) \geq \delta_{t_0/2} e^{-\rho_{t_0/2} [1+(t_0/2)^{-1}]|v|^2}$ for $t\geq t_0$, as desired.%. For $t > 2$, we can apply \eqref{e:lower_bound} with $t_0 = t-1$ and obtain 
\end{proof}

%\begin{remark}
%In the proof of Theorem \ref{t:lower}(b), if we replace $\xi(x-x_0)$ with $\xi(x-x_0-(t-t_0)v_0)$ in the definition of $\psi_2$, the term $2v\cdot \nabla \xi(x-x_0)$ in \eqref{e:Lpsi} would be replaced by $2(v-v_0)\cdot \nabla\xi(x-x_0-(t-t_0)v_0)$, and since $|v-v_0|\leq r$, this would allow us to choose a smaller value of $\mu$. This does not change the statement of Theorem \ref{t:lower}, but it demonstrates that lower bounds are propagated more efficiently along the characteristics $x = x_0 + (t-t_0) v_0$.
%\end{remark}
\appendix

\section{Dependence of local estimates on ellipticity constants}\label{s:A}

In \cite{golse2016}, a Harnack inequality and local $C^\alpha$ estimate are proven for kinetic Fokker-Planck equations of the form \eqref{e:FP}.  We are concerned only with the local $L^\infty$ estimate (Theorem \ref{t:local} above), which does not require the full strength of the Harnack inequality. In this appendix, we estimate the dependence of the constant on $\Lambda$, $\lambda$, and $\|B\|_{L^\infty}$. %(In \cite{golse2016}, it is assumed that $\|B\|_{L^\infty} \leq \Lambda$, but we do not make this assumption here.) 
The dependence on $\lambda$ and $\|B\|_{L^\infty}$ is not relevant for the present article, but may be useful to know in other contexts. For simplicity, we will assume $\Lambda$, $\|B\|_{L^\infty} \geq 1$ and $\lambda \leq 1$.

The proof of the $L^\infty$ estimate (Theorem 12 in \cite{golse2016}) proceeds in the following steps:
%\begin{enumerate}
%\item[1.] A global $L^2$ estimate for $\nabla_v g$, $D_x^{\frac 1 3}g$, and $D_t^{\frac 1 3} g$.
%
%\item[2.] A local energy estimate of Caccioppoli type.
%
%\item[3.] Combining the prior two steps to conclude local gain of integrability.
%
%\item[4.] Using De Giorgi's iteration technique, bootstrap the gain of integrability up to $L^\infty$.
%\end{enumerate}

\noindent \emph{Step 1: Global regularity estimate} (\cite[Lemma 10]{golse2016}). Starting with an equation of the form
\begin{equation}\label{e:lemma10}
(\partial_t + v\cdot \nabla_x) g = \nabla_v\cdot(A\nabla_v g) + \nabla_v\cdot H_1 + H_0,
\end{equation} 
with $H_0, H_1 \in L^2(\R^{2d+1})$ and $g,H_0,H_1$ supported in $(t,x,v) \in \R\times \R^d \times B_{r_0}(0)$, one integrates against $2g$, and using the Poincar\'e inequality in $v$, obtains the estimate
%\[\lambda \int_{\R^{2d+1}} |\nabla_v g|^2 \dd x \dd v \dd t \leq \frac C \lambda \left(\int_{\R^{2d+1}} |H_0|^2\]
\[\lambda \|\nabla_v g\|_{L^2}^2 \leq \frac C \lambda \left( \|H_0\|_{L^2}^2 +  \|H_1\|_{L^2}^2\right),\]
where $L^q=L^q(\R^{2d+1})$. Applying the hypoelliptic estimate of \cite{bouchut2002hypoelliptic}, and using $(1+|v|^2) \leq 1+r_0^2$ and $\|A\|_{L^\infty} \leq \Lambda$, gives
\[\|D_x^{\frac 1 3} g\|_{L^2}^2 + \|D_t^{\frac 1 3} g\|_{L^2}^2 \lesssim \Lambda (1+r_0^2)\left(\|\nabla_v g\|_{L^2}^2 + \|H_1\|_{L^2}^2 + \|H_0\|_{L^2}^2 \right), \]
 which, combined with the above estimate for $\|\nabla_v g\|_{L^2}$ and the Sobolev embedding $H^{\frac 1 3}(\R^{2d+1}) \subset L^p(\R^{2d+1})$, yields 
% \begin{equation*}%\label{e:global_est}
% \|D_x^{\frac 1 3} g\|_{L^2}^2 + \|D_t^{\frac 1 3} g\|_{L^2}^2 + \|\nabla_v g\|_{L^2}^2 \leq C\left( \|H_1\|_{L^2}^2 + \|H_0\|_{L^2}^2 \right), 
% \end{equation*}  
 \begin{equation}\label{e:global_est}
 \|g\|_{L^p}^2 \leq C\left(\|H_1\|_{L^2}^2 + \|H_0\|_{L^2}^2 \right), 
 \end{equation}
 with $p = 6(2d+1)/(6d+1)$ and the constant $C$ proportional to $\dfrac \Lambda {\lambda^2} (1+r_0^2)$.
\medskip

\noindent \emph{Step 2: Caccioppoli inequality} (\cite[Lemma 11]{golse2016}). Considering subsolutions of \eqref{e:FP} defined in a cylinder $Q_{r_1}$ and integrating the equation against $2g \Psi^2$ over $\mathcal R := [t_1,0]\times \R^{2d}$, with $t_1 \in (-r_1^2,0]$ and $\Psi$ a smooth, compactly supported cutoff with $0 \leq \Psi \leq 1$, one has (using Young's inequality)
\begin{align*}
\int_{\mathcal R} &\frac d {dt} (g^2 \Psi^2) + 2 \lambda \int_{\mathcal R} |\nabla_v g|^2 \Psi^2  \\
&\leq  \int_{\mathcal R} g^2 \left[ (\partial_t + v\cdot \nabla_x)(\Psi^2) + 8(\Lambda + \|B\|_{L^\infty})^2 \lambda^{-1} (|\nabla_v \Psi|^2 + \Psi^2) \right] + 2\int_{\mathcal R} g s \Psi^2 + \lambda \int_{\mathcal R} |\nabla_v g|^2 \Psi^2.
\end{align*}
For $r_0 \in (0,r_1)$, choose $\Psi$ such that $\Psi \equiv 1$ in $Q_{r_0}$, $\Psi(t,x,v) = 0$ for $t=0$, and $\mbox{supp} \,\Psi \subset Q_{r_1}$, and obtain
\begin{equation}\label{e:lemma11}
\|g\|_{L^\infty_t L^2_{x,v}(Q_{r_0})}^2 + \|\nabla_v g\|_{L^2(Q_{r_0})}^2 \leq C\left( C_{0,1}\|g\|_{L^2(Q_{r_1})}^2 + \|s\|_{L^2(Q_{r_1})}^2\right),
\end{equation}
with $C \lesssim \left(\dfrac {\Lambda + \|B\|_{L^\infty}} \lambda \right)^2$ and $C_{0,1}$ depending on $r_0$ and $r_1$ (via derivatives of $\Psi$).

\medskip

\noindent \emph{Step 3: Gain of integrability} (\cite[Theorem 6]{golse2016}). % For concentric cylinders $Q_{int} \subset Q_{mid} \subset Q_{ext}$, 
Letting $Q_{\intt} = Q_{r_0}$, $Q_{\ext} = Q_{r_1}$, and $Q_{\midd} = Q_{(r_0+r_1)/2}$, 
define cutoffs $\chi_1$ with $\chi_1 \equiv 1$ in $Q_{\intt}$ and $\chi_1 \equiv 0$ outside $Q_{\midd}$, and $\chi_{1/2}$ with $\chi_{1/2} \equiv 1$ in $Q_{\midd}$ and $\chi_{1/2} \equiv 0$ outside $Q_{\ext}$. For $g$ a nonnegative solution of \eqref{e:FP}, the truncated function $g\chi_1$ is a subsolution of \eqref{e:lemma10} with 
\begin{align*}
H_1 &= (-A\nabla_v \chi_1) g \chi_{1/2},\\
H_0 &= (B \chi_1 - A\nabla_v \chi_1)\cdot \nabla_v g \chi_{1/2} + g \chi_{1/2}(\partial_t + v\cdot \nabla_x) \chi_1 + s\chi_1.
\end{align*}
One has
\begin{align*}
\|H_1\|_{L^2}^2 &\lesssim \Lambda^2 \|\nabla_v \chi_1\|_{L^\infty}^2 \|g\|_{L^2(Q_{\ext})}^2,\\
\|H_0\|_{L^2}^2 &\lesssim (\|B\|_{L^\infty}^2 + \Lambda^2\|\nabla_v \chi_1\|_{L^\infty}^2) \|\nabla_v g\|_{L^2(Q_{\midd})}^2 + \|(\partial_t + v\cdot\nabla_x)\chi_1\|_{L^\infty}^2\|g\|_{L^2(Q_{\ext})}^2 + \|s\|_{L^2(Q_{\ext})}^2.
\end{align*}
Using \eqref{e:lemma11} to estimate $\|\nabla_v g\|_{L^2(Q_{\midd})}^2$, we have, after collecting terms,
\begin{align*}
\|H_0\|_{L^2}^2 + \|H_1\|_{L^2}^2 &\lesssim (\Lambda^2 + \|B\|_{L^\infty}^2)^2 \lambda^{-2}\Big[(1+\|\nabla_v \chi_1\|_{L^\infty}^2)\|s\|_{L^2}^2 \\
& \quad \left. + \left((1+ \|\nabla_v \chi_1\|_{L^\infty}^2)(1+ C_{0,1}) + \|(\partial_t + v\cdot\nabla_x)\chi_1\|_{L^\infty}^2\right) \|g\|_{L^2(Q_{\ext})}^2\right].
%\left[ \Lambda^2 \|\nabla_v \chi_1\|_{L^\infty}^2 (1 + (\Lambda + \|B\|_{L^\infty})^2 \lambda^{-2} C_{0,1}) + (\|B\|_{L^\infty} + \Lambda^2\|\nabla_v \chi_1\| \right] \|g\|_{L^2(Q_{ext})}^2
\end{align*}
Estimating $\|\nabla_v \chi_1\|_{L^\infty}$ in terms of $r_0$ and $r_1$,  and applying \eqref{e:global_est}, one has
\[\|g\|_{L^p(Q_{\intt})}^2 \leq C\left(C_{0,1}^2 \|g\|_{L^2(Q_{\ext})}^2 + C_{0,1} \|s\|_{L^2(Q_{\ext})}^2\right),\]
with $C \lesssim \Lambda (\Lambda^4 + \|B\|_{L^\infty}^4) \lambda^{-4}$ and $C_{0,1}$ as in Step 2.

\medskip

\noindent \emph{Step 4: De Giorgi iteration} (\cite[Theorem 12]{golse2016}). For any $\mathfrak g>0$, the goal is to show the existence of $\kappa\in (0,1]$ such that if $\|s\|_{L^\infty(Q_1)} \leq \mathfrak g$ and $\|g\|_{L^2(Q_1)} \leq \kappa$, then $\|g\|_{L^\infty(Q_{1/2})} \leq \frac 1 2$. Taking $\mathfrak g = 1$ and applying this result to $g/(\kappa^{-1}\|g\|_{L^2(Q_{1})}+\|s\|_{L^\infty(Q_1)})$ will imply the estimate \eqref{e:degiorgi} with constant proportional to $\kappa^{-1}$.

Define radii $r_n = \frac 1 2(1+2^{-n})$ and constants $C_n = \frac 1 2 (1-2^{-n})$ for all integers $n\geq 0$. %, and defines time-independent cutoffs $\Psi_n$ with $\Psi_n \equiv 1$ in $Q_{r_n}\cap\{t=0\}$ and $\Psi_n \equiv 0$ outside $Q_{r_{n-1}}\cap\{t=0\}$.
Considering $g_n = (g-C_n)_+$, which is a subsolution of \eqref{e:FP} in $Q_{r_n}$ with source term $s\chi_{\{g\geq C_n\}}$, and proceeding as in Step 2 with a suitable cutoff $\Psi_n$ in $Q_{r_n}$, one obtains
\[U_n:= \|g_n\|_{L^\infty_t L^2_{x,v}(Q_{r_n})}^2 \leq C \left(4^n \|g_n\|_{L^2(Q_{r_{n-1}})}^2 + 2 \int_{Q_{r_{n-1}}} g_n s\right),  \]
with $C \lesssim \left(\dfrac {\Lambda + \|B\|_{L^\infty}} \lambda \right)^2$. Applying H\"older's inequality in both terms on the right-hand side, and using the fact that $g_{n-1} \geq C_n - C_{n-1} = 2^{-n-1}$ whenever $g_n \geq 0$, one concludes
\begin{equation}\label{e:Un}
U_n \leq C 2^{4n} \left[ \|g_{n-1}\|_{L^p(Q_{r_{n-1}})}^2 U_{n-1}^{1-\frac 2 p} + \|s\|_{L^\infty(Q_{r_0})} \|g_{n-1}\|_{L^p(Q_{r_{n-1}})} U_{n-1}^{1-\frac 1 p}\right], 
\end{equation}
with the same $C$. Next, from Step 3 we have 
\[\|g_{n-1}\|_{L^p(Q_{r_{n-1}})}^2 \lesssim \Lambda (\Lambda^4 + \|B\|_{L^\infty}^4) \lambda^{-4} \left(8^n \|g_{n-1}\|_{L^2(Q_{r_{n-2}})}^2 + 4^n \int_{Q_{r_{n-2}}} s^2 \chi_{\{g_{n-1}\geq 0\}}\right) ,\]
with $p$ as above. Estimating both terms on the right-hand side in a manner similar to the above, we have
\[\|g_{n-1}\|_{L^p(Q_{r_{n-1}})}^2 \lesssim \Lambda (\Lambda^4 + \|B\|_{L^\infty}^4) \lambda^{-4} 2^{4n}(1+\mathfrak g^2) U_{n-2}.\]
%\lesssim \Lambda (\Lambda^4 + \|B\|_{L^\infty}^4) \lambda^{-4} \left(2^{3n} U_{n-2} + 4^n \int_{Q_{r_{n-2}}} s^2 \chi_{\{g_{n-1}\geq 0\}}\right) ,\]
Here, we take $\mathfrak g =1$. Plugging this into \eqref{e:Un} and using $U_{n-1} \leq U_{n-2}$ and $U_{n-2} \leq \kappa \leq 1$ gives 
\[ U_n \lesssim \Lambda(\Lambda^6 + \|B\|_{L^\infty}^6)\lambda^{-6}  2^{8n} U_{n-2}^{\frac 3 2 - \frac 1 p}.\]
%(We have used $U_{n-2}^{2-\frac 2 p} \leq U_{n-2}^{\frac 3 2 - \frac 1 p}$ since $U_{n-2} \leq 1$.)
Renaming $V_n = U_{2n}$, $\alpha = \frac 3 2 - \frac 1 p > 1$, and 
\[\beta =  \Lambda(\Lambda^6 + \|B\|_{L^\infty}^6)\lambda^{-6} 2^8,\]
we have $V_n \leq \beta^n V_{n-1}^\alpha$, which, applied iteratively, gives
\[ V_n \leq \beta^{n + \alpha(n-1) + \alpha^2(n-2) + \cdots + \alpha^{n-1}} V_0^{\alpha^n} \leq \left(\beta^{\frac \alpha {(\alpha-1)^2}} V_0\right)^{\alpha^n}.\]
If 
\[V_0 = \|g\|_{L^\infty_t L^2_{x,v}(Q_{r_0})}^2 \leq \beta^{\frac{-\alpha}{(\alpha-1)^2}} =: \kappa, \]
then $U_{2n} = V_n \to 0$ as $n\to \infty$, and $U_\infty = \|(g - \frac 1 2)_+ \|_{L^2(Q_{1/2})}^2 = 0$, so $g \leq \frac 1 2$ in $Q_{1/2}$. This implies the constant in the $L^\infty$ estimate is proportional to 
\[\kappa^{-1} = \beta^{\frac \alpha{(\alpha-1)^2}} \lesssim \left((\Lambda^7 + \Lambda \|B\|_{L^\infty}^6)\lambda^{-6}\right)^{3(6d+4)(2d+1)},\]
so we may take a value (not necessarily optimal) of $P = 21(6d+4)(2d+1)$ in Theorem \ref{t:local}. %since the value of $p$ implies $\alpha/(\alpha-1)^2 = 3(6d+4)(2d+1)$. 
\bibliographystyle{abbrv}
\bibliography{landau}
\end{document}